\documentclass{amsart}
\usepackage{amsmath,amssymb}
\usepackage[initials]{amsrefs}

\evensidemargin=\oddsidemargin
\newtheorem{thm}{Theorem}[section]

\newtheorem{lem}[thm]{Lemma}

\newtheorem{prop}[thm]{Proposition}

\newtheorem{remark}[thm]{Remark}

\numberwithin{equation}{section}

\newcommand{\mc}{\mathcal}

\newcommand{\A}{\mathbb{A}}

\newcommand{\bF}{\mathbb{F}}

\begin{document}
\title[Mean Values of $L$-functions and its Derivatives: IV]
{Mean values of derivatives of $L$-functions in function fields: IV}

\author[J. Andrade]{Julio Andrade}
\address{Department of Mathematics, University of Exeter, Exeter, EX4 4QF, United Kingdom}
\email{j.c.andrade@exeter.ac.uk}

\author[H. Jung]{Hwanyup Jung}
\address{\rm Department of Mathematics Education, Chungbuk National University, Cheongju 361-763, Korea}
\email{hyjung@chungbuk.ac.kr}

\subjclass[2010]{Primary 11M38; Secondary 11M06, 11G20, 11M50, 14G10}
\keywords{function fields, derivatives of $L$--functions, moments of $L$--functions, quadratic Dirichlet $L$--functions, random matrix theory}

\begin{abstract}
In this series, we investigate the calculation of mean values of derivatives of Dirichlet $L$-functions in function fields
using the analogue of the approximate functional equation and the Riemann Hypothesis for curves over finite fields.
The present paper generalizes the results obtained in the first paper.
For  $\mu\geq1$ an integer, we compute the mean value of the $\mu$-th derivative of quadratic Dirichlet $L$-functions over the rational function field.
We obtain the full polynomial in the asymptotic formulae for these mean values where we can see the arithmetic dependence
of the lower order terms that appears in the asymptotic expansion.
\end{abstract}
\date{\today}

\maketitle


\section{Introduction}

This is part 4 of a series of papers devoted to the study
of mean values of derivatives of $L$-functions in function fields.
The main method to compute such mean values is based on the use
of the approximate functional equation for function field $L$-functions
as first developed by Andrade and Keating in \cite{And-Kea}.
In part 1 \cite{AR}, we computed the full polynomial
in the asymptotic expansion of $\sum_{D\in\mathcal{H}}L^{''}(\tfrac{1}{2},\chi_{D})$,
where $L(s,\chi_{D})$ is the quadratic Dirichlet $L$-function
in function fields associated with the quadratic character $\chi_{D}$
where $D$ is monic and square-free and $\mathcal{H}$ is
the set of all monic and square-free polynomials of odd degree
in $\mathbb{F}_{q}[t]$. Now we generalize the results obtained in the first part of this series.

In this paper we establish
asymptotic formulas for the mean values of the $\mu$-th derivative
of Dirichlet $L$-functions associated to real quadratic function fields
and to imaginary quadratic function fields, in other words, we average
the derivatives of quadratic Dirichlet $L$-functions over monic and square-free
polynomials of odd degree and of even degree respectively.
The calculations carried out in this paper generalizes the previous
work \cite{AR} but the calculations here are much more subtle and lengthy.
Moreover, extra care is needed to bound the error terms and to obtain the full
polynomial in the asymptotic formulas.

Before we proceed and state the main results of this paper it is important
to remind the reader that the results of this paper can be seen as a function
field version of moments of derivatives of the Riemann zeta function as given
by Ingham \cite{Ing} and then further developed by the work of Conrey \cite{Con4},
Gonek \cite{Gon} and Conrey, Rubinstein and Snaith \cite{CRS}. A second motivation
of this work comes from the pioneering work of Hoffstein and Rosen
\cite{HR} about the study of mean values of Dirichlet $L$-functions in
function fields.


\section{A Short Background on Function Fields}

For the main notation used in this paper, we suggest the reader to consult the book by Rosen \cite{Ros},
the book by Thakur \cite{Tha} and the other papers in this series \cite{A1, A2, AR}.

Let $\bF_{q}$ be a finite field with $q$ elements, where $q$ is odd.
We denote by $\A = \bF_{q}[t]$ the polynomial ring over $\mathbb{F}_{q}$ and $\A^{+} = \{f \in \A : \text{ $f$ is monic}\}$.
We also use that $\A^{+}_{n} = \{f\in \A^{+} : \deg(f) = n\}$, $\A^{+}_{\le n} = \{f\in \A^{+} : \deg(f) \le n\}$, $\A^{+}_{> n} = \{f\in \A^{+} : \deg(f) > n\}$,
and that $\mc H = \{f \in\A^{+} : \text{ $f$ is square-free}\}$, with $\mc H_{n} = \mc H \cap \A^{+}_{n}$.

The zeta function attached to $\A$ is defined by the following Dirichlet series,
\begin{equation*}
\zeta_{\A}(s)=\sum_{f\in \A^{+}}\frac{1}{|f|^{s}} \quad \text{for Re$(s)>1$},
\end{equation*}
where $|f|=q^{\deg(f)}$ for $f\ne 0$ and $|f|=0$ for $f=0$. We can easily prove that
\begin{equation*}
\zeta_{\A}(s)=\frac{1}{1-q^{1-s}}.
\end{equation*}
The quadratic Dirichlet $L$-function of the rational function field $k=\bF_{q}(t)$ is defined to be
\begin{equation*}
L(s,\chi_{D})=\sum_{f\in \A^{+}}\frac{\chi_{D}(f)}{|f|^{s}} \qquad \text{for Re$(s)>1$},
\end{equation*}
where $\chi_{D}$ is the quadratic character defined by the quadratic residue symbol in $\mathbb{F}_{q}[t]$, i.e.,
\begin{equation*}
\chi_{D}(f)=\left(\frac{D}{f}\right),
\end{equation*}
and $D$ is a square-free monic polynomial.
In other words, if $P\in A$ is monic irreducible we have
\begin{equation}\label{2.4}
\chi_{D}(P)= \begin{cases}
0, & \text{ if } P\mid D,  \\
1, & \text{ if } \text{$P\nmid D$ and $D$ is a square modulo $P$},  \\
-1, & \text{ if } \text{$P \nmid D$ and $D$ is a non-square modulo $P$}.
\end{cases}
\end{equation}
For a more detailed discussion about Dirichlet characters for function fields see \cite[Chapter 3]{Ros} and \cite{And-Kea, Fai-Rud}.

In this paper we work with the family of quadratic Dirichlet $L$-functions that are associated to polynomials in $\mathcal{H}_{2g+1}$ and $\mathcal{H}_{2g+2}$.

If $D\in\mathcal{H}_{2g+1}$, the $L$-function associated to $\chi_{D}$ is the numerator of the zeta function associated to the hyperelliptic curve
defined by the affine equation $C_{D} : y^{2} = D(t)$ and, consequently, $L(s,\chi_{D})$ is a polynomial of degree $2g$ in the variable $u=q^{-s}$ given by
\begin{align}\label{2.7}
L(s,\chi_{D})&=\sum_{n=0}^{2g}A(n,\chi_{D})q^{-ns},
\end{align}
where
$$
A(n,\chi_{D})=\sum_{f\in\A^{+}_{n}}\chi_{D}(f).
$$
(see \cite[Propositions 14.6 and 17.7]{Ros} and \cite[Section 3]{And-Kea}).

This $L$-function, as it is expected,
satisfies a functional equation. Namely

\begin{equation}\label{2.8}
L(s,\chi_{D})=(q^{1-2s})^{g}L(1-s,\chi_{D}).
\end{equation}
The Riemann hypothesis for curves, proved by Weil \cite{Wei}, tell us that all the zeros of $L(s,\chi_{D})$ have real part equals $1/2$.

If $D\in\mathcal{H}_{2g+2}$ then the $L$-function is a polynomial of degree $2g+1$ and also satisfies a functional equation
and a Riemann Hypothesis, with the exception of having one zero with absolute value $1$.
For further details see \cite{Fai-Rud, Jun, Ros}.


\section{Statement of Results}

For any $D\in\mc H$, let $L(s,\chi_{D})$ be the quadratic Dirichlet $L$-function associated to $\chi_{D}$.
Let $\mu$ be a positive integer.
Let $L^{(\mu)}(s,\chi_{D})$ be the $\mu$-th derivative of $L(s,\chi_{D})$.
For any integer $n\ge 0$, let $J_{\mu}(n)$ be the sum of the $\mu$-th powers of the first $n$ positive integers,
i.e., $J_{\mu}(n) = \sum_{\ell=1}^{n} \ell^{\mu}$. Faulhaber's formula tell us that $J_{\mu}(n)$ can be rewritten
as a polynomial in $n$ of degree $\mu+1$ with zero constant term, that is, $J_{\mu}(n) = \sum_{m=1}^{\mu+1} j_{\mu}(m) n^{m}$.
The coefficients of this polynomial are related to Bernoulli numbers through the following formula known as Bernoulli's formula
\begin{align*}
J_{\mu}(n)=\frac{1}{\mu+1}\sum_{l=0}^{\mu}\binom{\mu+1}{l}B_{l}^{+}n^{\mu+1-l},
\end{align*}
where $\binom{\mu+1}{l}$ denotes the binomial coefficient and $B_{l}^{+}$ are the second Bernoulli numbers.

Let
\begin{align*}
G(s) = \sum_{L\in\A^{+}}\frac{\mu(L)}{|L|^{s}\prod_{P|L}(1+|P|)},
\end{align*}
where $\mu(L)$ is the M\"{o}bius function for polynomials.
So for any integer $m\ge 0$, we have
\begin{align}
\frac{G^{(m)}(s)}{(-\ln q)^{m}} = \sum_{L\in\A^{+}}\frac{\mu(L) \deg(L)^{m}}{|L|^{s}\prod_{P|L}(1+|P|)}.
\end{align}

We are now ready to state two of the main results of this paper.
The first theorem is the mean values of derivatives of Dirichlet $L$-functions
associated to the imaginary quadratic function field $k(\sqrt{D})$ with $D\in\mathcal{H}_{2g+1}$.

\begin{thm}\label{thm:1}
Let $\mu$ be a fixed positive integer and $q$ be an odd fixed integer. Then we have
\begin{align*}
\sum_{D\in\mc H_{2g+1}} \frac{L^{(\mu)}(\tfrac{1}{2},\chi_{D})}{(\ln q)^{\mu}}
&= \frac{(-2)^{\mu}|D|}{\zeta_{\A}(2)} \left(G(1)J_{\mu}([\tfrac{g}2]) + \frac{G^{(\mu)}(1)}{(-\ln q)^{\mu}}\right)  \\
&\hspace{1em}- \frac{(-2)^{\mu}|D|}{\zeta_{\A}(2)}\sum_{n=1}^{\mu+1} j_{\mu}(n)\frac{G^{(n)}(1)}{(-\ln q)^{n}}  \\
&\hspace{1em}+ \frac{2^{\mu}|D|}{\zeta_{\A}(2)} \sum_{m=0}^{\mu} {\mu\choose m}
(-g)^{\mu-m}\left(G(1)J_{m}([\tfrac{g-1}2]) + \frac{G^{(m)}(1)}{(-\ln q)^{m}}\right) \nonumber\\
&\hspace{1em}- \frac{2^{\mu}|D|}{\zeta_{\A}(2)} \sum_{m=0}^{\mu}{\mu\choose m} (-g)^{\mu-m} \sum_{n=1}^{m+1} j_{m}(n)\frac{G^{(n)}(1)}{(-\ln q)^{n}}
+ O(|D|^{\frac{7}8} (\log_{q}|D|)^{\mu}),
\end{align*}
where on the right hand side $|D|=q^{2g+1}$.
\end{thm}

The second theorem is the mean values of derivatives of Dirichlet $L$-functions
associated to the real quadratic function field $k(\sqrt{D})$ with $D\in\mathcal{H}_{2g+2}$.

\begin{thm}\label{thm:2}
Let $\mu$ be a positive integer and $q$ be an odd fixed integer.
Then we have
\begin{align*}
\sum_{D\in\mc H_{2g+2}}\frac{L^{(\mu)}(\tfrac{1}2,\chi_{D})}{(-\ln q)^{\mu}}
&= \frac{2^{\mu}|D|}{\zeta_{\A}(2)} \left(G(1) J_{\mu}([\tfrac{g}2]) + \frac{G^{(\mu)}(1)}{(-\ln q)^{\mu}}\right)  \\
&\hspace{1em}- \frac{2^{\mu}|D|}{\zeta_{\A}(2)}\sum_{n=1}^{\mu+1} j_{\mu}(n)\frac{G^{(n)}(1)}{(-\ln q)^{n}}- (g+1)^{\mu} G(1) |D| q^{[\frac{g}2]-\frac{g+1}{2}} \\
&\hspace{1em}+ \frac{2^{\mu}|D|}{\zeta_{\A}(2)}\sum_{a+b+c=\mu} \frac{(-1)^{c}\mu !}{a! b! c!} \frac{g^{a}\delta^{(b)}(\tfrac{1}{2})}{(-2\ln q)^{b}}
\left(G(1) J_{c}([\tfrac{g-1}2]) + \frac{G^{(c)}(1)}{(-\ln q)^{c}}\right) \nonumber  \\
&\hspace{1em} - \frac{2^{\mu}|D|}{\zeta_{\A}(2)}\sum_{a+b+c=\mu} \frac{(-1)^{c}\mu !}{a! b! c!} \frac{g^{a}\delta^{(b)}(\tfrac{1}{2})}{(-2\ln q)^{b}}
\sum_{n=1}^{c+1} j_{c}(n)\frac{G^{(n)}(1)}{(-\ln q)^{n}} \nonumber  \\
&\hspace{1em}-G(1) |D| q^{[\frac{g-1}2]-\frac{g}{2}}\sum_{m=0}^{\mu} {\mu\choose m} \frac{g^{\mu-m}\delta^{(m)}(\tfrac{1}{2})}{(-\ln q)^{m}}
+ O(|D|^{\frac{7}8} (\log_{q}|D|)^{\mu}),
\end{align*}
where $\delta(s) = \frac{1-q^{-s}}{1-q^{s-1}}$ and on the right hand side $|D|=q^{2g+2}$.
\end{thm}

\begin{remark}
Recent work of Florea \cite{Flo} on the first moment of quadratic Dirichlet $L$-functions over function fields lead us
to believe that the error term provided above is not optimal.
We will return to this topic in a future paper where we intend to use Florea's calculations to improve the error term above.
\end{remark}

\begin{remark}
As far as we checked results of this type are unknown for the family of quadratic Dirichlet $L$-functions
associated to the quadratic characters $\chi_{d}$ in the number field setting.
It should be possible to obtain the analogues of the results of this paper in the number field setting
by using the same technique as those employed by Jutila in \cite{Jut}.
The main difference would be on the size of the error term, where in the function fields
we can use the full power of the Riemann Hypothesis for curves to obtain unconditionally
better estimates than those in the number field setting.
\end{remark}


\section{Main Tools}
In this section we present a few auxiliary results that will be used in the proof of the main theorems.

\begin{lem}\label{Bound-lemma2-1-BBB}
If $f\in\A$ is a non-square polynomial, then
\begin{align*}
\bigg|\sum_{D\in\mc H_{n}}\chi_{D}(f)\bigg| \ll |D|^{\frac{1}2} |f|^{\frac{1}4}.
\end{align*}
\end{lem}
For a proof of Lemma \ref{Bound-lemma2-1-BBB} see \cite[Lemma 4.3]{And2} and \cite[Lemma 4.1]{And3}.
For a similar estimate see \cite[Lemma 3.1]{Fai-Rud}.

\begin{lem}\label{Lem-2-2-Lem-aaa}
Let $f \in \A^{+}$.
For any $\varepsilon>0$, we have
\begin{align*}
\sum_{\substack{D\in\mc H_{n}\\ (D,f)=1}} 1 = \frac{|D|}{\zeta_{\A}(2)} \prod_{P|f}(1+|P|^{-1})^{-1} + O(|D|^{\frac{1}2} |f|^{\varepsilon}).
\end{align*}
\end{lem}
\begin{proof}
This is Proposition 5.2 in \cite{And-Kea}.
\end{proof}

\begin{lem}\label{Lem-2-3-Lem-bbb}
We have
\begin{align*}
\sum_{L\in\A^{+}_{\ell}}\prod_{P|L}(1+|P|^{-1})^{-1} = q^{\ell} \sum_{L\in\A_{\le \ell}^{+}} \frac{\mu(L)}{|L|\prod_{P|L}(1+|P|)}.
\end{align*}
\end{lem}
\begin{proof}
This is Lemma 5.7 in \cite{And-Kea}.
\end{proof}

\begin{lem}\label{Lem-2-4-Lem-ccc}
Let $m\ge 0$ be an integer and $h=g$ or $h=g+1/2$.
Then we have
\begin{align*}
\sum_{\substack{L\in\A^{+}\\ \deg(L)> [\frac{h}2]}}\frac{\mu(L)\deg(L)^{m}}{|L|\prod_{P|L}(1+|P|)} = O(g^{m} q^{-\frac{g}2}).
\end{align*}
\end{lem}
\begin{proof}
This is Lemma 3.4 in \cite{AR}.
\end{proof}

\begin{lem}\label{mcMhm-mu-4.5}
For $h\in\{g-1, g\}$ and $m\in\{0,1,\ldots, \mu\}$, let
\begin{align*}
\mc M_{h,m}(\mu) = \frac{2^{m}|D|}{\zeta_{\A}(2)}\sum_{\ell=0}^{[\frac{h}2]} \ell^{m} q^{-\ell}\sum_{L\in\A^{+}_{\ell}}\prod_{P|L}(1+|P|^{-1})^{-1}.
\end{align*}
Then we have
\begin{align*}
\mc M_{h,m}(\mu) &= \frac{2^{m}|D|}{\zeta_{\A}(2)} \left(G(1) J_{m}([\tfrac{h}2]) + \frac{G^{(m)}(1)}{(-\ln q)^{m}}\right)   \\
&\hspace{1em}- \frac{2^{m}|D|}{\zeta_{\A}(2)}\sum_{a=1}^{m+1} j_{m}(a)\frac{G^{(a)}(1)}{(-\ln q)^{a}} + O(g^{m} q^{-\frac{g}2}|D|).
\end{align*}
\end{lem}
\begin{proof}
By Lemma \ref{Lem-2-3-Lem-bbb}, we can write
\begin{align}\label{eq4-1eq}
\mc M_{h,m}(\mu) &= \frac{2^{m}|D|}{\zeta_{\A}(2)}\sum_{\ell=0}^{[\frac{h}2]} \ell^{m} \sum_{L\in\A^{+}_{\le \ell}}\frac{\mu(L)}{|L|\prod_{P|L}(1+|P|)}\nonumber\\
&\hspace{1em}= \frac{2^{m}|D|}{\zeta_{\A}(2)}\sum_{L\in\A^{+}_{\le [\frac{h}2]}}\frac{\mu(L)}{|L|\prod_{P|L}(1+|P|)} \sum_{\deg(L)\le\ell\le [\frac{h}2]} \ell^{m}.
\end{align}
For integer $k \ge 0$, recall that $J_{k}(n) = \sum_{\ell=1}^{n} \ell^{k}$, which is a polynomial in $n$ of degree $k+1$ with zero constant term.
Write $J_{k}(n) = \sum_{m=1}^{k+1} j_{k}(m) n^{m}$.
Then we have
\begin{align}\label{mm4-2-mmm}
\sum_{\deg(L)\le\ell\le [\frac{h}2]} \ell^{m} &= J_{m}([\tfrac{h}2]) + \deg(L)^{m} - \sum_{a=1}^{m+1} j_{m}(a) \deg(L)^{a}.
\end{align}
Inserting \eqref{mm4-2-mmm} into \eqref{eq4-1eq}, we have
\begin{align}\label{eq4-3eq}
\mc M_{h,m}(\mu) &= \frac{2^{m}|D|}{\zeta_{\A}(2)}J_{m}([\tfrac{h}2])\sum_{L\in\A^{+}_{\le [\frac{h}2]}}\frac{\mu(L)}{|L|\prod_{P|L}(1+|P|)} \nonumber \\
&\hspace{1em}+ \frac{2^{m}|D|}{\zeta_{\A}(2)}\sum_{L\in\A^{+}_{\le [\frac{h}2]}}\frac{\mu(L)\deg(L)^{m}}{|L|\prod_{P|L}(1+|P|)} \nonumber \\
&\hspace{1em}- \frac{2^{m}|D|}{\zeta_{\A}(2)}\sum_{a=1}^{m+1} j_{m}(a)\sum_{L\in\A^{+}_{\le [\frac{h}2]}}\frac{\mu(L)\deg(L)^{a}}{|L|\prod_{P|L}(1+|P|)}.
\end{align}
Then, from \eqref{eq4-3eq}, by using Lemma \ref{Lem-2-4-Lem-ccc}, we get that
\begin{align}\label{eq4-4eq}
\mc M_{h,m}(\mu) &= \frac{2^{m}|D|}{\zeta_{\A}(2)}J_{m}([\tfrac{h}2])\sum_{L\in\A^{+}}\frac{\mu(L)}{|L|\prod_{P|L}(1+|P|)} \nonumber \\
&\hspace{1em}+ \frac{2^{m}|D|}{\zeta_{\A}(2)}\sum_{L\in\A^{+}}\frac{\mu(L)\deg(L)^{m}}{|L|\prod_{P|L}(1+|P|)} \nonumber \\
&\hspace{1em}- \frac{2^{m}|D|}{\zeta_{\A}(2)}\sum_{a=1}^{m+1} j_{m}(a)\sum_{L\in\A^{+}}\frac{\mu(L)\deg(L)^{a}}{|L|\prod_{P|L}(1+|P|)}
+ O(g^{m} q^{-\frac{g}2}|D|).
\end{align}
We also recall that for any integer $m\ge 0$, we have
\begin{align}\label{mm4-3-mmm}
\frac{G^{(m)}(s)}{(-\ln q)^{m}} = \sum_{L\in\A^{+}}\frac{\mu(L) \deg(L)^{m}}{|L|^{s}\prod_{P|L}(1+|P|)}.
\end{align}
Finally, by \eqref{eq4-4eq} and \eqref{mm4-3-mmm}, we get
\begin{align*}
\mc M_{h,m}(\mu) &= \frac{2^{m}|D|}{\zeta_{\A}(2)} \left(G(1) J_{m}([\tfrac{h}2]) + \frac{G^{(m)}(1)}{(-\ln q)^{m}}\right)   \\
&\hspace{1em}- \frac{2^{m}|D|}{\zeta_{\A}(2)}\sum_{a=1}^{m+1} j_{m}(a)\frac{G^{(a)}(1)}{(-\ln q)^{a}} + O(g^{m} q^{-\frac{g}2}|D|).
\end{align*}
\end{proof}

\begin{lem}\label{mcMhm-mu-4.6}
For $h\in\{g-1, g\}$, let
\begin{align*}
\mc N_{h}(\mu) = \frac{|D|}{\zeta_{\A}(2)}q^{-\frac{h+1}{2}}\sum_{\ell=0}^{[\frac{h}2]} \sum_{L\in\A^{+}_{\ell}}\prod_{P|L}(1+|P|^{-1})^{-1}.
\end{align*}
Then we have
\begin{align*}
\mc N_{h}(\mu) = G(1) |D| q^{[\frac{h}2]-\frac{h+1}{2}} + O(g|D|^{\frac{3}{4}}).
\end{align*}
\end{lem}
\begin{proof}
By Lemma \ref{Lem-2-3-Lem-bbb}, we can write
\begin{align*}
\mc N_{h}(\mu) &= \frac{|D|}{\zeta_{\A}(2)}q^{-\frac{h+1}{2}}\sum_{\ell=0}^{[\frac{h}2]} q^{\ell}
\sum_{L\in\A_{\le \ell}^{+}} \frac{\mu(L)}{|L|\prod_{P|L}(1+|P|)}\nonumber\\
&\hspace{1em}= \frac{|D|}{\zeta_{\A}(2)}q^{-\frac{h+1}{2}}\sum_{L\in\A^{+}_{\le [\frac{h}2]}}\frac{\mu(L)}{|L|\prod_{P|L}(1+|P|)}
\sum_{\deg(L)\le\ell\le [\frac{h}2]} q^{\ell}.
\end{align*}
Since
\begin{align*}
\sum_{\deg(L)\le\ell\le [\frac{h}2]} q^{\ell} = \zeta_{\A}(2) \left(q^{[\frac{h}2]} - q^{\deg(L)-1}\right),
\end{align*}
we get that
\begin{align}\label{eq4-6eq}
\mc N_{h}(\mu) = |D| q^{[\frac{h}2]-\frac{h+1}{2}}\sum_{L\in\A^{+}_{\le [\frac{h}2]}}\frac{\mu(L)}{|L|\prod_{P|L}(1+|P|)}
- |D| q^{-\frac{h+3}{2}} \sum_{L\in\A^{+}_{\le [\frac{h}2]}}\frac{\mu(L)}{\prod_{P|L}(1+|P|)}.
\end{align}
By Lemma \ref{Lem-2-4-Lem-ccc} and \eqref{mm4-3-mmm}, we have
\begin{align}\label{eq4-1eqrrtt4-7}
\sum_{L\in\A^{+}_{\le [\frac{h}2]}}\frac{\mu(L)}{|L|\prod_{P|L}(1+|P|)}
= \sum_{L\in\A^{+}}\frac{\mu(L)}{|L|\prod_{P|L}(1+|P|)} + O(q^{-\frac{g}2})
= G(1) + O(q^{-\frac{g}2}).
\end{align}
We also have
\begin{align}\label{eq4-1eqrrtt4-8}
\bigg|\sum_{L\in\A^{+}_{\le [\frac{h}2]}}\frac{\mu(L)}{\prod_{P|L}(1+|P|)}\bigg|
&\ll \sum_{\ell=0}^{[\frac{h}2]} \sum_{L\in\mc H_{\ell}}\frac{1}{|L|} \ll g.
\end{align}
By inserting \eqref{eq4-1eqrrtt4-7} and \eqref{eq4-1eqrrtt4-8} into \eqref{eq4-6eq}, we get
\begin{align*}
\mc N_{h}(\mu) = G(1) |D| q^{[\frac{h}2]-\frac{h+1}{2}} + O(g|D|^{\frac{3}{4}}).
\end{align*}
\end{proof}

\section{Proof of Theorem \ref{thm:1}}
In this section we give a proof of Theorem \ref{thm:1}.

\subsection{$\mu$-th derivative of $L(s,\chi_{D})$ for $D\in\mc H_{2g+1}$}
For any $D\in\mc H_{2g+1}$, the approximate functional equation for $L(s,\chi_{D})$ (\cite[Lemma 3.3]{And-Kea}) gives us
\begin{align}\label{L-ff}
L(s,\chi_{D}) = \sum_{f\in\A^{+}_{\le g}} \chi_{D}(f) |f|^{-s} + q^{(1-2s)g} \sum_{f\in\A^{+}_{\le g-1}} \chi_{D}(f) |f|^{s-1}.
\end{align}

\begin{lem}
Let $D \in \mc H_{2g+1}$.
For any integer $\mu \ge 0$, we have
\begin{align*}
\frac{L^{(\mu)}(s,\chi_{D})}{(\ln q)^{\mu}} = \sum_{n=0}^{g} (-n)^{\mu} A_{n}(D) q^{-ns}
+ q^{(1-2s)g} \sum_{m=0}^{\mu} {\mu\choose m} (-2g)^{\mu-m} \sum_{n=0}^{g-1} n^{m} A_{n}(D) q^{(s-1)n},
\end{align*}
where $A_{n}(D) = \sum_{f\in\A^{+}_{n}}\chi_{D}(f)$.
In particular, we also have
\begin{align}\label{eq1-1-eqii}
\frac{L^{(\mu)}(\tfrac{1}2,\chi_{D})}{(\ln q)^{\mu}} &= \sum_{n=0}^{g} (-n)^{\mu} A_{n}(D) q^{-\frac{n}2}
+ \sum_{m=0}^{\mu} {\mu\choose m} (-2g)^{\mu-m} \sum_{n=0}^{g-1} n^{m} A_{n}(D) q^{-\frac{n}2}.
\end{align}
\end{lem}
\begin{proof}
By \eqref{L-ff}, we can write $L(s,\chi_{D}) = \alpha(s) + \beta(s) \gamma(s)$, where
\begin{align*}
\alpha(s) = \sum_{n=0}^{g} A_{n}(D) q^{-ns}, ~~\beta(s) = q^{(1-2s)g} ~~\text{ and }~~ \gamma(s) = \sum_{n=0}^{g-1} A_{n}(D) q^{(s-1)n}.
\end{align*}
Then we have
\begin{align*}
L^{(\mu)}(s,\chi_{D}) = \alpha^{(\mu)}(s) + \sum_{m=0}^{\mu} {\mu\choose m} \beta^{(\mu-m)}(s) \gamma^{(m)}(s).
\end{align*}
For any integer $m \ge 0$, we also have
\begin{align*}
\alpha^{(m)}(s) = (\ln q)^{m}\sum_{n=0}^{g} (-n)^{m} A_{n}(D) q^{-ns}, ~~  \beta^{(m)}(s) = (-2g\ln q)^{m} q^{(1-2s)g}
\end{align*}
and
\begin{align*}
\gamma^{(m)}(s) = (\ln q)^{m} \sum_{n=0}^{g-1} n^{m} A_{n}(D) q^{(s-1)n}.
\end{align*}
By combining the previous equations it follows that
\begin{align*}
\frac{L^{(\mu)}(s,\chi_{D})}{(\ln q)^{\mu}} &= \sum_{n=0}^{g} (-n)^{\mu} A_{n}(D) q^{-ns}
+ q^{(1-2s)g} \sum_{m=0}^{\mu} {\mu\choose m} (-2g)^{\mu-m} \sum_{n=0}^{g-1} n^{m} A_{n}(D) q^{(s-1)n}
\end{align*}
and
\begin{align*}
\frac{L^{(\mu)}(\tfrac{1}2,\chi_{D})}{(\ln q)^{\mu}} = \sum_{n=0}^{g} (-n)^{\mu} A_{n}(D) q^{-\frac{n}2}
+ \sum_{m=0}^{\mu} {\mu\choose m} (-2g)^{\mu-m} \sum_{n=0}^{g-1} n^{m} A_{n}(D) q^{-\frac{n}2}.
\end{align*}
\end{proof}

Write
$$
\mc S_{h,m}^{\rm o}(\mu) = \sum_{n=0}^{h} n^{m} q^{-\frac{n}2}\sum_{f\in\A_{n}^{+}} \sum_{D\in\mc H_{2g+1}}\chi_{D}(f)
$$
for $h\in\{g-1, g\}$ and $m\in\{0,1,\ldots, \mu\}$.
Then, by \eqref{eq1-1-eqii}, we can write
\begin{align*}
\sum_{D\in\mc H_{2g+1}} \frac{L^{(\mu)}(\tfrac{1}{2},\chi_{D})}{(\ln q)^{\mu}}
= (-1)^{\mu} \mc S_{g,\mu}^{\rm o}(\mu) + \sum_{m=0}^{\mu} {\mu\choose m} (-2g)^{\mu-m} \mc S_{g-1,m}^{\rm o}(\mu).
\end{align*}

\subsection{Averaging $\mc S_{h,m}^{\rm o}(\mu)$}
In this subsection we obtain an asymptotic formula of $\mc S_{h,m}^{\rm o}(\mu)$.

\begin{prop}
For $h\in\{g-1, g\}$ and $m\in\{0,1,\ldots, \mu\}$, we have
\begin{align}\label{Shm-053}
\mc S_{h,m}^{\rm o}(\mu)
&= \frac{2^{m}|D|}{\zeta_{\A}(2)} \left(G(1) J_{m}([\tfrac{h}2]) + \frac{G^{(m)}(1)}{(-\ln q)^{m}}\right)  \nonumber \\
&\hspace{1em}- \frac{2^{m}|D|}{\zeta_{\A}(2)}\sum_{n=1}^{m+1} j_{m}(n)\frac{G^{(n)}(1)}{(-\ln q)^{n}}
+ O(|D|^{\frac{7}8} (\log_{q}|D|)^{m}).
\end{align}
\end{prop}
\begin{proof}
We split the sum over $f$ with $f$ being a perfect square of a polynomial or not.
Then we have
\begin{align*}
\mc S_{h,m}^{\rm o}(\mu) = \mc S_{h,m}^{\rm o}(\mu)_{\square} + \mc S_{h,m}^{\rm o}(\mu)_{\ne\square},
\end{align*}
where
\begin{align}\label{square-odd}
\mc S_{h,m}^{\rm o}(\mu)_{\square} = \sum_{n=0}^{h} n^{m} q^{-\frac{n}2}\sum_{\substack{f\in\A^{+}_{n}\\ f=\square}} \sum_{D\in\mc H_{2g+1}} \chi_{D}(f)
\end{align}
and
\begin{align}\label{nonsquare-odd}
\mc S_{h,m}^{\rm o}(\mu)_{\ne\square} = \sum_{n=0}^{h} n^{m} q^{-\frac{n}2}\sum_{\substack{f\in\A^{+}_{n}\\ f\ne\square}} \sum_{D\in\mc H_{2g+1}}\chi_{D}(f).
\end{align}
For the contribution of non-squares, from \eqref{nonsquare-odd} by using Lemma \ref{Bound-lemma2-1-BBB}, we have
\begin{align}\label{ns-odd}
|\mc S_{h,m}^{\rm o}(\mu)_{\ne\square}|
&\ll \sum_{n=0}^{h} n^{m} q^{-\frac{n}2}\sum_{\substack{f\in\A^{+}_{n}\\ f\ne\square}} \bigg|\sum_{D\in\mc H_{2g+1}}\chi_{D}(f)\bigg|  \nonumber\\
&\ll \sum_{n=0}^{h} n^{m} q^{-\frac{n}2}\sum_{f\in\A^{+}_{n}}|D|^{\frac{1}2} |f|^{\frac{1}4}
\ll |D|^{\frac{7}8} (\log_{q}|D|)^{m}.
\end{align}
Now, we consider the contribution of squares.
From \eqref{square-odd}, by using Lemma \ref{Lem-2-2-Lem-aaa}, we get
\begin{align}\label{Shm-057}
\mc S_{h,m}^{\rm o}(\mu)_{\square}
= \sum_{n=0}^{h} n^{m} q^{-\frac{n}2}\sum_{\substack{f\in\A^{+}_{n}\\ f=\square}}\sum_{\substack{D\in\mc H_{2g+1}\\ (D,f)=1}} 1
= \mc M_{h,m}(\mu) + O(g^{m} q^{g\varepsilon} |D|^{\frac{1}2}),
\end{align}
where
\begin{align*}
\mc M_{h,m}(\mu)
= \frac{2^{m}|D|}{\zeta_{\A}(2)}\sum_{\ell=0}^{[\frac{h}2]} \ell^{m} q^{-\ell} \sum_{L\in\A^{+}_{\ell}}\prod_{P|L}(1+|P|^{-1})^{-1}.
\end{align*}
By Lemma \ref{mcMhm-mu-4.5}, we have
\begin{align}\label{Mhm-0511}
\mc M_{h,m}(\mu) &= \frac{2^{m}|D|}{\zeta_{\A}(2)} \left(G(1) J_{m}([\tfrac{h}2]) + \frac{G^{(m)}(1)}{(-\ln q)^{m}}\right)  \nonumber\\
&\hspace{1em}- \frac{2^{m}|D|}{\zeta_{\A}(2)}\sum_{n=1}^{m+1} j_{m}(n)\frac{G^{(n)}(1)}{(-\ln q)^{n}} + O(g^{m} q^{-\frac{g}2}|D|).
\end{align}
By inserting \eqref{Mhm-0511} into \eqref{Shm-057}, it follows that
\begin{align}\label{Shm-0512}
\mc S_{h,m}^{\rm o}(\mu)_{\square} &= \frac{2^{m}|D|}{\zeta_{\A}(2)} \left(G(1) J_{m}([\tfrac{h}2]) + \frac{G^{(m)}(1)}{(-\ln q)^{m}}\right)  \nonumber\\
&\hspace{1em}- \frac{2^{m}|D|}{\zeta_{\A}(2)}\sum_{n=1}^{m+1} j_{m}(n)\frac{G^{(n)}(1)}{(-\ln q)^{n}} + O(g^{m} q^{-\frac{g}2}|D|).
\end{align}
Finally, combining \eqref{ns-odd} and \eqref{Shm-0512}, we obtain the result.
\end{proof}

\subsection{Completing the proof}
Recall that
\begin{align}\label{sumDH2g+1}
\sum_{D\in\mc H_{2g+1}} \frac{L^{(\mu)}(\tfrac{1}{2},\chi_{D})}{(\ln q)^{\mu}}
= (-1)^{\mu} \mc S_{g,\mu}^{\rm o}(\mu) + \sum_{m=0}^{\mu} {\mu\choose m} (-2g)^{\mu-m} \mc S_{g-1,m}^{\rm o}(\mu).
\end{align}
By \eqref{Shm-053} with $h=g$ and $m=\mu$, we have that
\begin{align}\label{Shm-053-mu-0}
(-1)^{\mu} \mc S_{g,\mu}^{\rm o}(\mu) &= \frac{(-2)^{\mu}|D|}{\zeta_{\A}(2)} \left(G(1)J_{\mu}([\tfrac{g}2]) + \frac{G^{(\mu)}(1)}{(-\ln q)^{\mu}}\right) \nonumber\\
&\hspace{1em}- \frac{(-2)^{\mu}|D|}{\zeta_{\A}(2)}\sum_{n=1}^{\mu+1} j_{\mu}(n)\frac{G^{(n)}(1)}{(-\ln q)^{n}}
+ O(|D|^{\frac{7}8} (\log_{q}|D|)^{\mu}).
\end{align}
We also, by \eqref{Shm-053}, have that
\begin{align}\label{Shm-053-5-11}
\sum_{m=0}^{\mu} {\mu\choose m} (-2g)^{\mu-m} \mc S_{g-1,m}^{\rm o}(\mu)
&= \frac{2^{\mu}|D|}{\zeta_{\A}(2)} \sum_{m=0}^{\mu} {\mu\choose m} (-g)^{\mu-m}\left(G(1)J_{m}([\tfrac{g-1}2]) + \frac{G^{(m)}(1)}{(-\ln q)^{m}}\right) \nonumber\\
&\hspace{1em}- \frac{2^{\mu}|D|}{\zeta_{\A}(2)} \sum_{m=0}^{\mu}{\mu\choose m} (-g)^{\mu-m} \sum_{n=1}^{m+1} j_{m}(n)\frac{G^{(n)}(1)}{(-\ln q)^{n}}
+ O(|D|^{\frac{7}8} (\log_{q}|D|)^{\mu}).
\end{align}
By inserting \eqref{Shm-053-mu-0} and \eqref{Shm-053-5-11} into \eqref{sumDH2g+1}, we complete the proof.
\begin{flushright}
$\square$
\end{flushright}

\section{Proof of Theorem \ref{thm:2}}
In this section we give a proof of Theorem \ref{thm:2}.

\subsection{$\mu$-th derivative of $L(s,\chi_{D})$ for $D\in\mc H_{2g+2}$}
For $D\in\mc H_{2g+2}$, the approximate functional equation for $L(s,\chi_{D})$ (\cite[Lemma 2.1]{Jun}) gives us
\begin{align}\label{LsD-even}
L(s,\chi_{D}) &= \sum_{f\in\A^{+}_{\le g}} \chi_{D}(f) |f|^{-s}- q^{-(g+1)s} \sum_{f\in\A^{+}_{\le g}} \chi_{D}(f) \nonumber  \\
&\hspace{1em} + q^{(1-2s)g} \delta(s) \sum_{f\in\A^{+}_{\le g-1}}\chi_{D}(f) |f|^{s-1} - q^{-gs} \delta(s)\sum_{f\in\A^{+}_{\le g-1}}\chi_{D}(f),
\end{align}
where $\delta(s) = \frac{1-q^{-s}}{1-q^{s-1}}$.

\begin{lem}
Let $D \in \mc H_{2g+2}$.
For any integer $\mu \ge 0$, we have
\begin{align*}
\frac{L^{(\mu)}(s,\chi_{D})}{(-\ln q)^{\mu}} &= \sum_{n=0}^{g} n^{\mu} A_{n}(D) q^{-ns} - (g+1)^{\mu} q^{-(g+1)s} \sum_{n=0}^{g} A_{n}(D) \\
&\hspace{1em} + q^{(1-2s)g} \sum_{a+b+c=\mu} \frac{\mu !}{a! b! c!} \frac{(2g)^{a} \delta^{(b)}(s)}{(-\ln q)^{b}} \sum_{n=0}^{g-1} (-n)^{c} A_{n}(D) q^{n(s-1)} \\
&\hspace{1em} - q^{-gs} \sum_{m=0}^{\mu} {\mu\choose m} \frac{g^{\mu-m} \delta^{(m)}(s)}{(-\ln q)^{m}} \sum_{n=0}^{g-1} A_{n}(D),
\end{align*}
where $A_{n}(D) = \sum_{f\in\A^{+}_{n}}\chi_{D}(f)$.
In particular, we also have
\begin{align}\label{eq-derk-eee}
\frac{L^{(\mu)}(\tfrac{1}2,\chi_{D})}{(-\ln q)^{\mu}}
&= \sum_{n=0}^{g} n^{\mu} A_{n}(D) q^{-\frac{n}{2}} - (g+1)^{\mu} q^{-\frac{g+1}{2}} \sum_{n=0}^{g} A_{n}(D) \nonumber \\
&\hspace{1em} + \sum_{a+b+c=\mu} \frac{\mu !}{a! b! c!} \frac{(2g)^{a} \delta^{(b)}(\frac{1}2)}{(-\ln q)^{b}}
\sum_{n=0}^{g-1} (-n)^{c} A_{n}(D) q^{-\frac{n}{2}} \nonumber \\
&\hspace{1em} - q^{-\frac{g}{2}} \sum_{m=0}^{\mu} {\mu\choose m} \frac{g^{\mu-m} \delta^{(m)}(\tfrac{1}{2})}{(-\ln q)^{m}} \sum_{n=0}^{g-1} A_{n}(D).
\end{align}
\end{lem}
\begin{proof}
By \eqref{LsD-even}, we can write
$$
L(s,\chi_{D}) = L_{1}(s) - L_{2}(s) + L_{3}(s) - L_{4}(s),
$$
where
\begin{align*}
L_{1}(s) &= \sum_{n=0}^{g} A_{n}(D)q^{-ns}, \\
L_{2}(s) &= q^{-(g+1)s} \sum_{n=0}^{g} A_{n}(D),  \\
L_{3}(s) &= q^{(1-2s)g} \delta(s)\sum_{n=0}^{g-1} A_{n}(D) q^{n(s-1)}, \\
L_{4}(s) &= q^{-gs} \delta(s)\sum_{n=0}^{g-1} A_{n}(D).
\end{align*}
For any integer $\mu \ge 0$, we have
\begin{align*}
\frac{L_{1}^{(\mu)}(s)}{(-\ln q)^{\mu}} &= \sum_{n=0}^{g} n^{\mu} A_{n}(D) q^{-ns}, \\
\frac{L_{2}^{(\mu)}(s)}{(-\ln q)^{\mu}} &= (g+1)^{\mu} q^{-(g+1)s} \sum_{n=0}^{g} A_{n}(D),  \\
\frac{L_{3}^{(\mu)}(s)}{(-\ln q)^{\mu}} &= q^{(1-2s)g} \sum_{a+b+c=\mu} \frac{\mu !}{a! b! c!} \frac{(2g)^{a} \delta^{(b)}(s)}{(-\ln q)^{b}} \sum_{n=0}^{g-1} (-n)^{c} A_{n}(D) q^{n(s-1)},   \\
\frac{L_{4}^{(\mu)}(s)}{(-\ln q)^{\mu}} &= q^{-gs} \sum_{m=0}^{\mu} {\mu\choose m} \frac{g^{\mu-m} \delta^{(m)}(s)}{(-\ln q)^{m}} \sum_{n=0}^{g-1} A_{n}(D).
\end{align*}
Hence, we get
\begin{align*}
\frac{L^{(\mu)}(s,\chi_{D})}{(-\ln q)^{\mu}} &= \sum_{n=0}^{g} n^{\mu} A_{n}(D) q^{-ns} - (g+1)^{\mu} q^{-(g+1)s} \sum_{n=0}^{g} A_{n}(D) \\
&\hspace{1em} + q^{(1-2s)g} \sum_{a+b+c=\mu} \frac{\mu !}{a! b! c!} \frac{(2g)^{a} \delta^{(b)}(s)}{(-\ln q)^{b}} \sum_{n=0}^{g-1} (-n)^{c} A_{n}(D) q^{n(s-1)} \\
&\hspace{1em} - q^{-gs} \sum_{m=0}^{\mu} {\mu\choose m} \frac{g^{\mu-m} \delta^{(m)}(s)}{(-\ln q)^{m}} \sum_{n=0}^{g-1} A_{n}(D).
\end{align*}
In particular, for $s=\frac{1}2$, it follows that
\begin{align*}
\frac{L^{(\mu)}(\tfrac{1}2,\chi_{D})}{(-\ln q)^{\mu}} &=  \sum_{n=0}^{g} n^{\mu} A_{n}(D) q^{-\frac{n}{2}} - (g+1)^{\mu} q^{-\frac{g+1}{2}} \sum_{n=0}^{g} A_{n}(D) \\
&\hspace{1em} + \sum_{a+b+c=\mu} \frac{\mu !}{a! b! c!} \frac{(2g)^{a} \delta^{(b)}(\frac{1}2)}{(-\ln q)^{b}} \sum_{n=0}^{g-1} (-n)^{c} A_{n}(D) q^{-\frac{n}{2}} \\
&\hspace{1em} - q^{-\frac{g}{2}} \sum_{m=0}^{\mu} {\mu\choose m} \frac{g^{\mu-m} \delta^{(m)}(\tfrac{1}{2})}{(-\ln q)^{m}} \sum_{n=0}^{g-1} A_{n}(D).
\end{align*}
\end{proof}

Write
\begin{align*}
\mc S_{h,m}^{\rm e}(\mu) = \sum_{n=0}^{h} n^{m} q^{-\frac{n}2}\sum_{f\in\A_{n}^{+}} \sum_{D\in\mc H_{2g+2}}\chi_{D}(f)
\end{align*}
and
\begin{align*}
{\mc T}_{h}(\mu) = q^{-\frac{h+1}{2}}\sum_{n=0}^{h}\sum_{f\in\A_{n}^{+}} \sum_{D\in\mc H_{2g+2}}\chi_{D}(f)
\end{align*}
for $h\in\{g-1, g\}$ and $m\in\{0,1,\ldots, \mu\}$.
Then, by \eqref{eq-derk-eee}, we can write
\begin{align*}
\sum_{D\in\mc H_{2g+2}}\frac{L^{(\mu)}(\tfrac{1}2,\chi_{D})}{(-\ln q)^{\mu}}
&= \mc S_{g,\mu}^{\rm e}(\mu) - (g+1)^{\mu} {\mc T}_{g}(\mu) \\
&\hspace{1em}+ \sum_{a+b+c=\mu} \frac{(-1)^{c}\mu !}{a! b! c!} \frac{(2g)^{a}\delta^{(b)}(\tfrac{1}{2})}{(-\ln q)^{b}}{\mc S}_{g-1,c}^{\rm e}(\mu) \\
&\hspace{1em}- \sum_{m=0}^{\mu} {\mu\choose m} \frac{g^{\mu-m}\delta^{(m)}(\tfrac{1}{2})}{(-\ln q)^{m}} {\mc T}_{g-1}(\mu).
\end{align*}

\subsection{Averaging $\mc S_{h,m}^{\rm e}(\mu)$}
In this subsection we obtain an asymptotic formula of $\mc S_{h,m}^{\rm e}(\mu)$.

\begin{prop}
For $h\in\{g-1, g\}$ and $m\in\{0,1,\ldots, \mu\}$, we have
\begin{align}\label{mcShm-063}
\mc S_{h,m}^{\rm e}(\mu)
&= \frac{2^{m}|D|}{\zeta_{\A}(2)} \left(G(1) J_{m}([\tfrac{h}2]) + \frac{G^{(m)}(1)}{(-\ln q)^{m}}\right) \nonumber  \\
&\hspace{1em}- \frac{2^{m}|D|}{\zeta_{\A}(2)}\sum_{a=1}^{m+1} j_{m}(a)\frac{G^{(a)}(1)}{(-\ln q)^{a}} + O(|D|^{\frac{7}8} (\log_{q}|D|)^{m}).
\end{align}
\end{prop}
\begin{proof}
We can write $\mc S_{h,m}^{\rm e}(\mu) = \mc S_{h,m}^{\rm e}(\mu)_{\square} + \mc S_{h,m}^{\rm e}(\mu)_{\ne\square}$, where
\begin{align}\label{mcShm-square-063}
\mc S_{h,m}^{\rm e}(\mu)_{\square} = \sum_{n=0}^{h} n^{m} q^{-\frac{n}{2}} \sum_{\substack{f\in\A^{+}_{n}\\ f=\square}} \sum_{D\in\mc H_{2g+2}} \chi_{D}(f)
\end{align}
and
\begin{align}\label{mcShm-nsquare-064}
\mc S_{h,m}^{\rm e}(\mu)_{\ne\square} = \sum_{n=0}^{h} n^{m} q^{-\frac{n}{2}} \sum_{\substack{f\in\A^{+}_{n}\\ f\ne\square}} \sum_{D\in\mc H_{2g+2}} \chi_{D}(f).
\end{align}
For the contribution of non-squares, from \eqref{mcShm-nsquare-064} by using Lemma \ref{Bound-lemma2-1-BBB}, we have
\begin{align}\label{nsquare-056}
|\mc S_{h,m}^{\rm e}(\mu)_{\ne\square}|
&\ll \sum_{n=0}^{h} n^{m} q^{-\frac{n}2} \sum_{\substack{f\in\A^{+}_{n}\\ f\ne\square}} \bigg|\sum_{D\in\mc H_{2g+2}} \chi_{D}(f)\bigg|\nonumber \\
&\ll \sum_{n=0}^{h} n^{m} q^{-\frac{n}2} \sum_{f\in\A^{+}_{n}}|D|^{\frac{1}2} |f|^{\frac{1}4}\nonumber \\
&\ll |D|^{\frac{7}8} (\log_{q}|D|)^{m}.
\end{align}
Now, we consider the contribution of squares.
From \eqref{mcShm-square-063}, by using Lemma 4.2, we get
\begin{align}\label{mcShm-square-066}
\mc S_{h,m}^{\rm e}(\mu)_{\square}
= \sum_{n=0}^{h} n^{m} q^{-\frac{n}2} \sum_{\substack{f\in\A^{+}_{n}\\ f=\square}} \sum_{\substack{D\in\mc H_{2g+2}\\ (D,f)=1}} 1
= \mc M_{h,m}(\mu) + O(g^{m} q^{g\varepsilon} |D|^{\frac{1}2}),
\end{align}
where
\begin{align*}
\mc M_{h,m}(\mu) = \frac{2^{m}|D|}{\zeta_{\A}(2)}\sum_{\ell=0}^{[\frac{h}2]} \ell^{m} q^{-\ell}\sum_{L\in\A^{+}_{\ell}}\prod_{P|L}(1+|P|^{-1})^{-1}.
\end{align*}
By Lemma \ref{mcMhm-mu-4.5}, we have
\begin{align}\label{mcMhm-mu-6.7}
\mc M_{h,m}(\mu) &= \frac{2^{m}|D|}{\zeta_{\A}(2)} \left(G(1) J_{m}([\tfrac{h}2]) + \frac{G^{(m)}(1)}{(-\ln q)^{m}}\right) \nonumber  \\
&\hspace{1em}- \frac{2^{m}|D|}{\zeta_{\A}(2)}\sum_{a=1}^{m+1} j_{m}(a)\frac{G^{(a)}(1)}{(-\ln q)^{a}} + O(g^{m} q^{-\frac{g}2}|D|).
\end{align}
By inserting \eqref{mcMhm-mu-6.7} into \eqref{mcShm-square-066}, it follows that
\begin{align}\label{mcShm-square-068}
\mc S_{h,m}^{\rm e}(\mu)_{\square}
&= \frac{2^{m}|D|}{\zeta_{\A}(2)} \left(G(1) J_{m}([\tfrac{h}2]) + \frac{G^{(m)}(1)}{(-\ln q)^{m}}\right) \nonumber  \\
&\hspace{1em}- \frac{2^{m}|D|}{\zeta_{\A}(2)}\sum_{a=1}^{m+1} j_{m}(a)\frac{G^{(a)}(1)}{(-\ln q)^{a}} + O(g^{m} q^{-\frac{g}2}|D|).
\end{align}
Finally, combining \eqref{nsquare-056} and \eqref{mcShm-square-068}, we obtain the result.
\end{proof}

\subsection{Averaging ${\mc T}_{h}(\mu)$}
In this subsection we obtain an asymptotic formula of ${\mc T}_{h}(\mu)$.

\begin{lem}
For $h\in\{g-1, g\}$, we have
\begin{align}\label{lem6.3-eq}
{\mc T}_{h}(\mu) = G(1) |D| q^{[\frac{h}2]-\frac{h+1}{2}} + O(|D|^{\frac{7}8} (\log_{q}|D|)).
\end{align}
\end{lem}
\begin{proof}
We can write ${\mc T}_{h}(\mu) = {\mc T}_{h}(\mu)_{\square} + {\mc T}_{h}(\mu)_{\ne\square}$, where
\begin{align}\label{mcThmu-square-610}
{\mc T}_{h}(\mu)_{\square} = q^{-\frac{h+1}{2}}\sum_{n=0}^{h} \sum_{\substack{f\in\A^{+}_{n}\\ f=\square}} \sum_{D\in\mc H_{2g+2}} \chi_{D}(f)
\end{align}
and
\begin{align}\label{mcThmu-nsquare-611}
{\mc T}_{h}(\mu)_{\ne\square} = q^{-\frac{h+1}{2}}\sum_{n=0}^{h} \sum_{\substack{f\in\A^{+}_{n}\\ f\ne\square}} \sum_{D\in\mc H_{2g+2}} \chi_{D}(f).
\end{align}
For the contribution of non-squares, from \eqref{mcThmu-nsquare-611} by using Lemma \ref{Bound-lemma2-1-BBB}, we have
\begin{align}\label{mcThmu-nsquare-612}
|{\mc T}_{h}(\mu)_{\ne\square}|
&\ll q^{-\frac{h+1}{2}}\sum_{n=0}^{h} \sum_{\substack{f\in\A^{+}_{n}\\ f\ne\square}} \bigg|\sum_{D\in\mc H_{2g+2}} \chi_{D}(f)\bigg| \nonumber  \\
&\ll q^{-\frac{h+1}{2}}\sum_{n=0}^{h} \sum_{f\in\A^{+}_{n}}|D|^{\frac{1}2} |f|^{\frac{1}4} \ll |D|^{\frac{7}8} (\log_{q}|D|).
\end{align}
Now, we consider the contribution of squares.
From \eqref{mcThmu-square-610}, by using Lemma \ref{Lem-2-2-Lem-aaa}, we get
\begin{align}\label{eq6-13eqq}
{\mc T}_{h}(\mu)_{\square} = q^{-\frac{h+1}{2}}\sum_{n=0}^{h} \sum_{\substack{f\in\A^{+}_{n}\\ f=\square}} \sum_{\substack{D\in\mc H_{2g+2}\\ (D,f)=1}} 1
= \mc N_{h}(\mu) + O(g q^{g\varepsilon} |D|^{\frac{1}2}),
\end{align}
where
$$
\mc N_{h}(\mu) = \frac{|D|}{\zeta_{\A}(2)}q^{-\frac{h+1}{2}}\sum_{\ell=0}^{[\frac{h}2]} \sum_{L\in\A^{+}_{\ell}}\prod_{P|L}(1+|P|^{-1})^{-1}.
$$
By Lemma \ref{mcMhm-mu-4.6}, we have
\begin{align}\label{eq6-14eqq}
\mc N_{h}(\mu) = G(1) |D| q^{[\frac{h}2]-\frac{h+1}{2}} + O(g|D|^{\frac{3}{4}}).
\end{align}
By inserting \eqref{eq6-14eqq} into \eqref{eq6-13eqq}, we get that
\begin{align}\label{eq6-15eqq}
{\mc T}_{h}(\mu)_{\square} = G(1) |D| q^{[\frac{h}2]-\frac{h+1}{2}} + O(g|D|^{\frac{3}{4}}).
\end{align}
Finally, combining \eqref{eq6-13eqq} and \eqref{eq6-15eqq}, we obtain the result.
\end{proof}

\subsection{Completing the Proof}
Recall that
\begin{align}\label{eq6-17eq}
\sum_{D\in\mc H_{2g+2}}\frac{L^{(\mu)}(\tfrac{1}2,\chi_{D})}{(-\ln q)^{\mu}}
&= \mc S_{g,\mu}^{\rm e}(\mu) - (g+1)^{\mu} {\mc T}_{g}(\mu) \nonumber\\
&\hspace{1em}+ \sum_{a+b+c=\mu} \frac{(-1)^{c}\mu !}{a! b! c!} \frac{(2g)^{a}\delta^{(b)}(\tfrac{1}{2})}{(-\ln q)^{b}}{\mc S}_{g-1,c}^{\rm e}(\mu)\nonumber\\
&\hspace{1em}- \sum_{m=0}^{\mu} {\mu\choose m} \frac{g^{\mu-m}\delta^{(m)}(\tfrac{1}{2})}{(-\ln q)^{m}} {\mc T}_{g-1}(\mu).
\end{align}
Then, by \eqref{mcShm-063}, we have
\begin{align}\label{mcShm-618}
\mc S_{g,\mu}^{\rm e}(\mu)
&= \frac{2^{\mu}|D|}{\zeta_{\A}(2)} \left(G(1) J_{\mu}([\tfrac{g}2]) + \frac{G^{(\mu)}(1)}{(-\ln q)^{\mu}}\right) \nonumber  \\
&\hspace{1em}- \frac{2^{\mu}|D|}{\zeta_{\A}(2)}\sum_{n=1}^{\mu+1} j_{\mu}(n)\frac{G^{(n)}(1)}{(-\ln q)^{n}} + O(|D|^{\frac{7}8} (\log_{q}|D|)^{\mu})
\end{align}
and
\begin{align}\label{mcShm-619}
&\sum_{a+b+c=\mu} \frac{(-1)^{c}\mu !}{a! b! c!} \frac{(2g)^{a}\delta^{(b)}(\tfrac{1}{2})}{(-\ln q)^{b}}{\mc S}_{g-1,c}^{\rm e}(\mu) \nonumber  \\
&\hspace{2em}= \frac{2^{\mu}|D|}{\zeta_{\A}(2)}\sum_{a+b+c=\mu} \frac{(-1)^{c}\mu !}{a! b! c!} \frac{g^{a}\delta^{(b)}(\tfrac{1}{2})}{(-2\ln q)^{b}}
\left(G(1) J_{c}([\tfrac{g-1}2]) + \frac{G^{(c)}(1)}{(-\ln q)^{c}}\right) \nonumber  \\
&\hspace{3em} - \frac{2^{\mu}|D|}{\zeta_{\A}(2)}\sum_{a+b+c=\mu} \frac{(-1)^{c}\mu !}{a! b! c!} \frac{g^{a}\delta^{(b)}(\tfrac{1}{2})}{(-2\ln q)^{b}}
\sum_{n=1}^{c+1} j_{c}(n)\frac{G^{(n)}(1)}{(-\ln q)^{n}} \nonumber  \\
&\hspace{3em}+ O(|D|^{\frac{7}8} (\log_{q}|D|)^{\mu}).
\end{align}
Now, by \eqref{lem6.3-eq}, we also have
\begin{align}\label{mcShm-620}
(g+1)^{\mu} {\mc T}_{g}(\mu) = (g+1)^{\mu} G(1) |D| q^{[\frac{g}2]-\frac{g+1}{2}} + O(|D|^{\frac{7}8} (\log_{q}|D|)^{\mu})
\end{align}
and
\begin{align}\label{mcShm-621}
&\sum_{m=0}^{\mu} {\mu\choose m} \frac{g^{\mu-m}\delta^{(m)}(\tfrac{1}{2})}{(-\ln q)^{m}} {\mc T}_{g-1}(\mu)  \nonumber  \\
&\hspace{1em}= G(1) |D| q^{[\frac{g-1}2]-\frac{g}{2}}\sum_{m=0}^{\mu} {\mu\choose m} \frac{g^{\mu-m}\delta^{(m)}(\tfrac{1}{2})}{(-\ln q)^{m}}
+ O(|D|^{\frac{7}8} (\log_{q}|D|)^{\mu}).
\end{align}
By inserting \eqref{mcShm-618}, \eqref{mcShm-619}, \eqref{mcShm-620} and \eqref{mcShm-621} into \eqref{eq6-17eq}, we complete the proof.

\section{A Remark on the derivatives of $G(s)$}
Recall that
\begin{align*}
G(s) = \sum_{L\in\A^{+}}\frac{\mu(L)}{|L|^{s}\prod_{P|L}(1+|P|)}.
\end{align*}
By expressing $G(s)$ as an Euler product, we obtain that
\begin{align}\label{eq-GS-eqq}
G(s) = \prod_{P}\left(1 - \frac{1}{|P|^{s}(1+|P|)}\right).
\end{align}
Taking logarithmic derivative in \eqref{eq-GS-eqq}, we have
\begin{align}\label{Der-1-Der}
\frac{G'(s)}{-\ln q} = G(s) H_{1}(s),
\end{align}
where
$$
H_{1}(s) = \sum_{P}\frac{\deg(P)}{1-|P|^{s}(|P|+1)}.
$$
For any integer $n\ge 1$, let
$$
H_{n+1}(s) = \frac{H'_{n}(s)}{-\ln q}.
$$
Let $\mc G$ be the set of finite sums of finite product of $G(s), H_{1}(s), H_{2}(s), \ldots$.
Define an operator $\varphi$ by $\varphi(G(s)) = G(s) H_{1}(s)$ and $\varphi(H_{n}(s)) = H_{n+1}(s)$ for any $n\ge 1$.
Extend $\varphi$ to $\mc G$ by the rules that $\varphi(A+B) = \varphi(A) + \varphi(B)$ and $\varphi(AB) = \varphi(A)B+A\varphi(B)$ for any $A, B \in \mc G$.
For example,
\begin{align*}
\varphi(G(s)) &= G(s) H_{1}(s) = \frac{G'(s)}{-\ln q}, \\
\varphi^{2}(G(s)) &= G(s) \left(H_{1}(s)^2 + H_{2}(s)\right) = \frac{G''(s)}{(-\ln q)^2}, \\
\varphi^{3}(G(s)) &= G(s)(H_{1}(s)^3 + 3 H_{1}(s) H_{2}(s) + H_{3}(s)) = \frac{G^{(3)}(s)}{(-\ln q)^3}.
\end{align*}
Inductively, we can show that
\begin{align}
\varphi^{n}(G(s)) = \frac{G^{(n)}(s)}{(-\ln q)^n}
\end{align}
for any integer $n\ge 1$.

For any monic irreducible polynomial $P$, let $f_{P}(s) = |P|^s(|P|+1)$.
We present below the function $H_{n}(s)$ for a few values of $n$ as this illustratesthe complexity of this function.
\begin{align*}
H_{1}(s) &= \sum_{P}\frac{\deg(P)}{1-f_{P}(s)},  \\
H_{2}(s) &= - \sum_{P} \frac{\deg(P)^2 f_{P}(s)}{(1 - f_{P}(s))^2}, \\
H_{3}(s) &= \sum_{P} \frac{\deg(P)^3 f_{P}(s)(1+ f_{P}(s))}{(1 - f_{P}(s))^3}, \\
H_{4}(s) &=  - \sum_{P} \frac{\deg(P)^4 f_{P}(s)(1+4f_{P}(s)+f_{P}(s)^2)}{(1 - f_{P}(s))^4}, \\
H_{5}(s) &= \sum_{P} \frac{\deg(P)^5 f_{P}(s)(1+ 11f_{P}(s)+ 11f_{P}(s)^2+f_{P}(s)^3)}{(1 - f_{P}(s))^5}.
\end{align*}







\end{document}